\documentclass[10pt,a4paper]{article}
\usepackage{latexsym,amsthm,amssymb,amscd,amsmath}
\newcounter{alphthm}
\theoremstyle{plain}
\newtheorem{thm}{Theorem}[section]

 \newtheorem{exam}[thm]{Example}
 
 \theoremstyle{definition}
 \newtheorem{defn}[thm]{Definition}
 \theoremstyle{remark}
 \newtheorem{rem}[thm]{\bf Remark}
 \numberwithin{equation}{subsection}

\begin{document}

\title {A Fixed-Point Theorem For Mapping Satisfying
a General Contractive Condition Of Integral Type Depended an
Another Function  }

\author{ S. Moradi\footnote{First author}  and A. Beiranvand \\\\
Faculty  of Science, Department of Mathematics\\
Arak University, Arak,  Iran\\
\date{}
 }
 \maketitle
\begin{abstract}
In this paper, we study the existence of fixed points for
 mappings defined on complete  metric space ($X,d$)
 satisfying a general contractive inequality of integral type depended
 on another function. This conditions is analogous of Banach
 conditions and Branciari Theorem.

\end{abstract}


\section{Introduction}

The first important result on fixed points for contractive-type
 mapping was the well-known Banach's Contraction Principle
 appeared in explicit form in Banach's thesis in 1922, where it was
 used to establish the existence of a solution for an integral
 equation [1]. In the general setting of complete metric space this
  theorem runs as follows(see[5,Theorem 2.1]\\
  or[10,Theorem1.2.2]).

\begin{thm}$($\textbf{Banach's Contraction Principle}$)$
 Let $(X,d)$ be a complete metric space and $f:X \longrightarrow
 X$ be a contraction $($there exists $k \in (0,1)$ such that for
 each $x,y \in X$; $d(fx,fy) \leq kd(x,y)$$)$. Then $f$ has a unique
 fixed point in $X$, and for each $x_{0} \in X$ the sequence of
 iterates $\{f^{n}x_{0}\}$ converges to this fixed point.

\end{thm}

After this classical result Kannan in [4] analyzed a substantially
new type of contractive condition. Since then there have been many
theorems dealing with mappings satisfying various types of
contractive inequalities. Such conditions involve linear and
nonlinear expressions (rational, irrational, and of general type).
The intrested reader who wants to know more about this matter is
recommended to go deep into the survey articles by Rhoades [7,8,9]
and Meszaros [6], and into the references therein. Another result
on fixed points for contractive-type mapping is generally
attributed to Edelstein (1962) who actually obtained slightly more
general versions. In the general setting of compact metric spaces
this result runs as followes (see [5, Theorem 2.2]).

\begin{thm}
 Let $(X,d)$ be a compact metric space and $f:X \longrightarrow
 X$ be a contractive $($for every $x,y \in X$ such that $x \neq y$; $d(fx,fy) < d(x,y)$$)$. Then $f$ has a unique
 fixed point in $X$, and for any $x_{0} \in X$ the sequence of
 iterates $\{f^{n}x_{0}\}$ converges to this fixed point.

\end{thm}

Also in 2002 in [3]  A. Branciari analyzed the existence of fixed
point for mapping $f$ defined on a complete metric space $(X,d)$
satisfying a contractive condition of integral type.(see the
following theorem).

\begin{thm}
Let $(X,d)$ be a complete metric space, $ \alpha \in (0,1)$ and
$f:X \longrightarrow X$ be a mapping such that for each $x,y \in
X$, $\int_0^{d(fx,fy)} \phi(t)  dt \leq \alpha{\int _0^{d(x,y)}
\phi(t)  dt}$, where $\phi: [0,+\infty)\longrightarrow
[0,+\infty)$ is a Lebesgue-integrable mapping which is summable
(i.e., with finite integral) on each compact subset of
$[0,+\infty)$, nonnegative, and such that for each $\epsilon > 0
,\int _0^\epsilon \phi(t) dt > 0$; then $f$ has a unique fixed
point $a\in X$ such that for each $x\in X$, $\underset
{n\rightarrow\infty} \lim f^{n} x=a$.

\end{thm}

The aim of this paper is to study the existence of fixed point for
mapping $f$ defined on a compact metric space$(X,d)$ such that is
$T_{\int \phi}-contraction$. In particular, we extend the main
theorem due to A. Branciari [3] (Theorem 1.3) and the main theorem
in [2] (2008). First we introduce the $T_{\int \phi}-contraction$
function and then extended the A.Branciari Theorem and the main
theorem in [2] and Banach-contraction principle, by the same
metod for proof of the A. Branciari Theorem. At the end of paper
some examples and applications concerning this kind of
contractions. In [3] A. Branciari gave an example (Example 3.6)
such that we can conclude this example by theorem 1.2. (because
$X=\{1/ n :n\in \Bbb N\}\bigcup \{0\} $, with metric induced by $
\Bbb R$, $ d(x,y)=|x-y|$, is a compact metric space and $f$ is a
contractive mapping). In the end of this paper we give an example
(Example 3.5) such that we can not conclude this example by
Theorem 1.1, Theorem 1.2. Branciari Theorem and the main theorem
in [2], but we can conclude this example by the main theorem
(Theorem 2.5 ) in this paper. In the sequel, $\Bbb N$ will
represent the set of natural numbers, $\Bbb R$ the set of real
number and $\Bbb R^{+}$ the set of nonnegative real number.


\section{Definitions and Main Result}

The following theorem (Theorem 2.5) is the main result of this
paper. In the first, we define some new definitions.

\begin{defn}
 Let $(X,d)$ be a metric space and $f,T:X \longrightarrow X$ be
 two functions and  $\phi:  [0,+\infty) \longrightarrow
[0,+\infty)$ be a Lebesgue-integrable mapping. A mapping $f$ is
said to be a $T_{\int \phi}-contraction$ if there exists $\alpha
\in (0,1)$ such that for all $x,y \in X $\\
$$ \int_0^{d(Tfx,Tfy)} \phi(t)  dt \leq \alpha{\int
_0^{d(Tx,Ty)} \phi(t)  dt}$$
\end{defn}
\begin{rem}
By taking $Tx=x$ and $\phi=1$, $T_{\int \phi}-contraction$ and
contraction are equivalent. Also by taking $Tx=x$ we can define
$\int \phi-contraction$.
\end{rem}
\begin{exam}
Let $X=[1,+\infty)$ with metric induced by $\Bbb{R}$:
$d(x,y)=|x-y|$. We consider two mappings $T,f:X \longrightarrow X$
by $Tx=\frac{1}{x}+1$ and $fx=2x$. Obviously $f$ is not
contraction but $f$ is $T_{\int{1}} -contraction$.

\end{exam}

\begin{defn}$[2]$
 Let $(X,d)$ be a metric space. A mapping $T:X \longrightarrow X$
 is said sequentially convergent if we have, for every sequence
 $\{y_{n}\}$, if $\{Ty_{n}\}$ is convergence then $\{y_{n}\}$ also is
 convergence. $T$ is said subsequentially convergent if we have, for
 every sequence $\{y_{n}\}$, if $\{Ty_{n}\}$ is convergence then
 $\{y_{n}\}$ has a convergent subsequence.
\end{defn}

\begin{thm}$[$Main theorem $]$
 Let $(X,d)$ be a complete metric space, $\alpha \in (0,1)$,  $T,f:X \longrightarrow
 X$ be mapping such that $T$ is continuous,  one-to-one and  subsequentially convergent and
 $f$ is $T_{\int \phi} -contraction$ where $\phi: [0,+\infty) \longrightarrow
[0,+\infty)$ is a Lebesgue-integrable mapping which is summable on
each compact subset of $[0,+\infty)$, nonnegative and such that
for each $\epsilon > 0,  \int _0^\epsilon \phi(t) dt >0 $; then
$f$ has a unique fixed point $a\in X$. Also if $T$ is sequentially
convergent, then for each $x_0\in X$, the sequence of iterates
$\{f^nx_{0} \}$ converges to this fixed point.

\end{thm}
\begin{proof}

\textbf{STEP 1}. Let $\alpha \in (0,1)$ such that for all $x,y\in
X$
\\ \\
$$\int_0^{d(Tfx,Tfy)} \phi(t)  dt \leq \alpha{\int _0^{d(Tx,Ty)}
\phi(t)  dt}. \hspace{2cm} (2.1)$$\\
  So if for $a,b> 0$, $\int_0^a \phi(t)  dt \leq \alpha{\int _0^b\phi(t)
dt}$ then $a<b.$\\ \\
\textbf{STEP 2}. We show that $f$ is a continuouse mapping.\\
 If $\underset {n\rightarrow \infty} \lim x_n=x$
 then by $\int_0^{d(Tfx_n , Tfx)} \phi(t)  dt \leq \alpha{\int _0^{d(Tx_n ,
Tx)}\phi(t) dt}$ and $\underset {n\rightarrow \infty}  \lim
d(Tx_n , Tx)=0$, we conclude that:\\
$$\lim _{n\rightarrow \infty}d(Tfx_n, Tfx)=0.$$ Since $T$ is
subsequentially convergent, $\{fx_n\}$ has a subsequence such
 $\{{fx_n}_k\}_{k=1}^\infty$  converge to a  $y\in X$. So
 $d(Ty,Tfx)=0$. Since $T$ is one-to-one, $y=fx$. Hence,
$\{fx_n\}$ has a subsequence converge to $fx$. \\
Therefore for every sequence $\{x_n\}$ converge to $x$, the
sequence $\{fx_n\}$ has a subsequence converge to $fx$. This shows
that $f$ is continuouse at $x$.\\ \\
\textbf{STEP 3}. Since (2.1)
is holds, for all $ n \in \Bbb N:$
   $$ \int _0^{d(Tf^{n+1} x , Tf^n x)} \phi(t) dt \leq
 \alpha^n {\int _0^{d(Tfx , Tx) }\phi(t) dt}\qquad \forall x\in X .$$
 As a consequence, since  $\alpha \in (0,1)$, we further have
           $$\int_0^{d(Tf^{n+1}x , Tf^nx)} \phi(t) dt \rightarrow  0^+ \hspace{0.5cm}
           as \hspace{0.5cm}
 n\rightarrow \infty \hspace {2cm} (2.2) $$
  Since\\
  $$ \hspace{0.5cm} \int _0^{\epsilon} \phi(t) dt > 0, \hspace{0.5cm}
\forall \epsilon> 0 \hspace{2cm}  (2.3)$$
is holds we conclude that\\
            $$\lim _{n\rightarrow \infty} d(Tf^{n+1}x , Tf^nx)=0
            \hspace{3cm}(2.4)$$
\textbf{Step 4}.  $\{Tf^nx\}$  is a bounded sequence. \\
If  $\{Tf^nx\} _{n=1}^\infty$is not a bounded sequence then, we
choose the sequence $\{n_k\}_{k=1}^\infty$  such that  $n_1=1$ and
for each $k\in \Bbb {N }$, $n_{k+1}$  is "minimal"  in the sense
that\\
  $$d(Tf^{n_{k+1}}x , Tf^{n_k}x) > 1.$$  \\
 So,
\begin{eqnarray*}
  1
  &<& d(Tf^{n_{k+1}}x  , Tf^{n_k}x)\\
  &\leq& d(Tf^{n_{k+1}}x , Tf^{n_{k+1}-1}x)  +d(Tf^{n_{k+1}-1}x , Tf^{n_k}x) \\
  &\leq& d(Tf^{n_{k+1}}x ,  Tf^{n_{k+1}-1}x)+1. \hspace{2.5cm} (2.5)
\end{eqnarray*}
  Hence, by (2.4) and (2.5) we conclude that
  $$d(Tf^{n_{k+1}}x , Tf^{n_k}x) \rightarrow 1 \hspace{0.5cm} as \hspace{0.5cm} k\rightarrow
  \infty  \hspace{2cm}   (2.6)$$\\
Also by step 1, \\
 $$d(Tf^{n_{k+1}}x , Tf^{n_k+1}x) \leq
  d(Tf^{n_{k+1}-1}x , Tf^{n_k}x). $$
Therefore,
\begin{eqnarray*}
 1-d(Tf^{n_k+1}x , Tf^{n_k}x)
  &<&d(Tf^{n_{k+1}}x ,Tf^{n_k}x)-d(Tf^{n_k+1}x , Tf^{n_k}x)\\
  &\leq&d(Tf^{n_{k+1}}x , Tf^{n_k+1}x)\\
 &\leq& d(Tf^{n_{k+1}-1}x,Tf^{n_k}x)\\
 &\leq& 1.
\end{eqnarray*}
 Hence, by (2.4),
\[d(Tf^{n_{k+1}}x , Tf^{n_k+1}x) \rightarrow
1\hspace{0.5cm}as\hspace{0.5cm}k\rightarrow \infty. \hspace{2cm}
(2.7)\] Therefore,
\begin{eqnarray*}
\int_0^{d(Tf^{n_{k+1}}x , Tf^{n_k+1}x)} \phi(t) dt
 &\leq &\alpha{\int_0^{d(Tf^{n_{k+1}-1}x , Tf^{n_{k}}x)}\phi(t) dt}\\
 &\leq& \alpha{\int_0^{1} \phi(t) dt}. \hspace{2cm} (2.8)
\end{eqnarray*}
By (2.7) and (2.8) we conclude that\\
\begin{eqnarray*}
\int_0^1 \phi(t) dt & =& \lim_{k\rightarrow \infty}
\int_0^{d(Tf^{n_{k+1}}x,
Tf^{n_k+1}x)} \phi(t) dt\\
&\leq& \alpha{\int_0^1 \phi(t) dt} .\\
\end{eqnarray*}
 So $\int_0^1\phi(t) dt=0$
and this is contradiction.\\
\\ \\
\\ \\
 \textbf{STEP 5}. By (2.1) for every $m,n\in
\Bbb{N} (m > n) $,

\[ \int_0^{d(Tf^{m}x , Tf^{n}x)} \phi(t) dt  \leq
\alpha^{n}{\int_0^{d(Tf^{m-n}x , Tx)}\phi(t) dt}. \hspace{2cm}
(2.9)\] By step 4, (2.9) and $\alpha\in (0,1),$
$$\underset {m,n\rightarrow\infty}\lim \int_0^{d(Tf^{m}x , Tf^{n}x)}
=0\hspace{2cm} (2.10)$$ Since (2.3) is hold $\underset
{m,n\rightarrow \infty}\lim d(Tf^{m}x ,Tf^{n}x)=0$, and this
shows that $\{Tf^{n}x\}_{n=1}^\infty$ is a
Cauchy sequence. Hence there exists $a\in X$ such that\\
$$\underset {n \rightarrow \infty} \lim  Tf^{n}x=a \hspace{2cm}(2.11)$$
\textbf{STEP 6}. Since $T$ is a subsequentially convergent,
$\{f^{n}x\}$ has a convergent subsequence. \\
So there exists $b\in X$ and $\{n_k\} _{k=1}^\infty$ such that
 $\underset {k \rightarrow \infty}\lim  f^{n_k}x=b$. Since $T$ is continuouse
$\underset {k\rightarrow \infty}\lim  Tf^{n_k}x=Tb$, and by (2.11)
we conclude that\\
$$ Tb=a. \hspace{2cm}(2.12)$$
Since $f$ is continuouse (step 2) and $\underset {k \rightarrow
\infty}\lim f^{n_k}x=b , \underset {k \rightarrow \infty}\lim
f^{n_{k}+1}x=fb$ and so
$\underset {k \rightarrow \infty} \lim  Tf^{n_k+1}x=Tfb. $\\
Again by (2.11) we have $$\underset {k \rightarrow \infty}\lim
Tf^{n_k+1}x=a$$ and therefore, $Tfb=a.$ So by (2.12), $Tfb=Tb.$
Since $T$ is one-to
one, $fb=b.$ Therefore $f$ has a fixed point.\\
\textbf{STEP 7}. Since $T$ is one-to-one and $f$ is $T_{\int
\phi}-contraction $, $f$ has a unique fixed point.
\end{proof}

\section{Examples and Applications}
In this section, we give some applications and some examples
concerning these contractive mapping of integral type, which
clarify the connection between our result and the classical
ones.\\
\begin{rem}
Theorem 2.5 is a generalization of the Banach's contraction
principle (Theorem 1.1), letting $\phi(t)=1$ for each $t \geq 0$
and $Tx=x$ for each $x\in X$ in Theorem 2.5, we have \\
\begin{eqnarray*}
\int_0^{d(Tfx , Tfy)}\phi(t) dt &=&d(fx , fy) \\
&\leq& \alpha d(x,y)\\
&=& \alpha{\int _0^{d(Tx , Ty)} \phi(t) dt}\\
\end{eqnarray*}
\end{rem}
\begin{rem}
Theorem 2.5 is a generalization of the A. Branciari theorem
(Theorem
1.3), letting $Tx=x$ for each $x\in X$ in Theorem 2.5, so \\
\begin{eqnarray*}
\int_0^{d(Tfx , Tfy)} \phi(t) dt
 &=& \int_0^{d(fx , fy)}\phi(t) dt\\
 &\leq&  \alpha {\int_0^{d(x,y)} \phi(t) dt }\\
 &=& \alpha{\int_0^{d(Tx , Ty)} \phi(t) dt}.
\end{eqnarray*}
\end{rem}
We can conclude the following theorem ( the main Theorem in [2])
by Theorem 2.5. \\

\begin{thm}
Let $(X,d)$ be a complete metric space and $ T:X\longrightarrow X$
be a one-to-one, continuouse and subsequentially convergent
mapping. Then for every $T-contraction$ function
$f:X\longrightarrow X$, $f$ has a unique fixed point. Also if $T$
is sequentially convergent, then for each $x_{0}\in X$, the
sequence of iterates $\{f^{n}x\}$ converges to this fixed point.
($f:X\longrightarrow X$ is $T-contraction$ if there exist $\alpha
\in (0,1)$ such that for all $x,y\in X $\\
$$d(Tfx , Tfy) \leq \alpha {d(Tx , Ty)}.)$$
\end{thm}
\begin{proof}
By taking $\phi(t)=1$ for each $t\in [0,+\infty)$ in Theorem 2.5
we can conclude this theorem.\\
\end{proof}
\begin{exam}
Let $X=[1,+\infty)$ with metric induced by $ \Bbb R
:d(x,y)=|x-y|,$ thus, since $X$ is a closed subset of $\Bbb R,$ it
is a complete metric space. we define $T,f:X\longrightarrow X$ by
$Tx=\ln{x}+1$ and $fx=k\sqrt{x}$ such that $k \geq 1$ be a fixed
element of $\Bbb R.$ Obviousely $f$ is not contraction, but $f$ is
$T_{\int1}-contraction$ and $T$ is one-to-one, continuouse and
sequentially convergent. So $f$ has a unique fixed point by
Theorem 2.5.
\end{exam}

The following example is the main example of this paper. In the
following we show that, we can not conclude this example by
Theorem 1.1, Theorem 1.2, Theorem 1.3 (Branciari Theorem) and
Theorem 3.3.

\begin{exam}
Let $X:={\{\frac{1}{n} \ \  | \ \ n\in \Bbb N \}} \bigcup {\{0}\}$
with metric induced by $\Bbb R :d(x,y):=|x-y|$, thus, since $X$ is
a closed subset of $\Bbb R$, it is a complete metric space. We
consider a mapping $f:X\longrightarrow X$ defined by
$$fx= \left\{\begin{array} {ll} \frac {1}{n+3} & ; x=\frac {1}{n},\: n \: is \: odd \\
0 & ; x=0\\
\frac {1}{n-1} & ; x=\frac{1}{n},\: n \: is \: even
\end{array}\right. $$ and defined
$\phi:[0,+\infty)\longrightarrow [0,+\infty)$ by\\
$$ \phi(t)=\left\{\begin{array} {ll} t^{{\frac {1}{t}}-2}[1-\log{t}] &
; t>0\\ 0 & ; t=0 \end{array}\right. $$\\
we have $\int_0^\tau \phi(t) dt=\tau ^{\frac{1}{\tau}}.$\\
By taking $n=2$ and $m=4$, $|f(1/m)-f(1/n)| > |1/m-1/n|$, so $f$
is not contraction and contractive. Hence, we can not conclude
that, $f$ has a fixed point by Theorem 1.1 and Theorem 1.2.\\
 Now we show that we can not use Branciari Theorem for this
example. For $x=1/m$, $y=1/n$ where $m$ and $n$ are even if \\
\\ \\
$$ \int_0^{|fx-fy|}\phi(t) dt \leq \alpha{\int_0^{|x-y|}\phi(t)
dt}$$ \\
then\\
$$|\frac {1}{m-1}-\frac {1}{n-1}|^{\frac{1}{|\frac {1}{m-1}- \frac {1}{n-1}|}} \leq
\alpha{|\frac{1}{m}- \frac {1}{n}|^{\frac {1}{|\frac{1}{m}-\frac
{1}{n}|}}}$$\\

$$ \Rightarrow \hspace{1cm} |\frac {m-n}{(m-1)(n-1)}|^{|\frac {(m-1)(n-1)}{m-n}|}
\leq \alpha{|\frac {m-n}{mn}|^{|\frac {mn}{m-n}|}}$$\\
For $m=4$ and $n=2$ we conclude that $1 < \alpha.$ So we can not
use Branciari Theorem.\\
Now we defined $T:X \longrightarrow X$ by\\
$$ Tx=\left\{\begin{array}{ll} \frac {1}{n-1} &; x=\frac{1}{n},\: n \: is \: even\\
 0 &; x=0 \\ \frac {1}{n+1} &;
x=\frac {1}{n}, \: n \: is \: odd \end{array}\right. $$\\
 Obviously $T$ is one-to-one and sequentially convergent and
 continuouse.\\
 we have
 $$ Tfx=\left\{\begin{array}{ll} \frac {1}{n+2} &; x=\frac{1}{n},\: n \:is \: odd\\ 0 &; x=0 \\ \frac {1}{n}
 &; x=\frac {1}{n},\: n\: is\: even \end{array}\right.$$\\
 \\ \\
 Since $\sup{\frac {|Tfx-Tfy|}{|Tx-Ty}|}=1$, $f$ is not
 $T-contraction$, and so we can not use Theorem 3.3 for this
 example. Now we show that the condition of Theorem 2.5 are holds.
 We show that $f$ is $T_{\int\phi}-contraction $ and \\
$$ \int_0^{|Tfx-Tfy|}\phi(t) dt \leq {\frac{1}{2}\int
_0^{|Tx-Ty|}\phi(t)
 dt} \hspace{0.5cm} for\ \ all\ \ x,y\in X. \hspace{2cm}(2.13)$$\\
\textbf{ Case 1}. Let $x=\frac{1}{m}, y=\frac{1}{n}$ and $m$ and
$n$ are even. Then

\begin{align*}
 &\qquad\qquad \int_0^{|Tfx-Tfy|}\phi(t) dt \leq {\frac{1}{2}\int_0^{|Tx-Ty|}\phi(t)
 dt} \\ &
 \Leftrightarrow\hspace{1cm}|\frac{1}{m}-\frac{1}{n}|^{\frac{1}{|\frac{1}{m}-\frac{1}{n}|}}\leq\frac
 {1}{2}{|\frac{1}{m-1}-\frac{1}{n-1}|^{\frac{1}{|\frac{1}{m-1}-\frac{1}{n-1}|}}} \\
 & \Leftrightarrow \hspace{1cm}|\frac{m-n}{mn}|^{|\frac{mn}{m-n}|}.
 |\frac{(m-1)(n-1)}{m-n}|^{|\frac{(m-1)(n-1)}{m-n}|}
 \leq\frac {1}{2}\\
 & \Leftrightarrow \hspace{1cm}
 |\frac{(m-1)(n-1)}{mn}|^{|\frac{(m-1)(n-1)}{m-n}|}.
 |\frac{(m-n)}{mn}|^{|\frac{(m+n-1)}{m-n}|} \leq\frac {1}{2}
\end{align*}
Obviously the last inequality is holds, because \\
$$|\frac{(m-1)(n-1)}{mn}| \leq 1 \:\: and \:\: |\frac{(m-1)(n-1)}{m-n}| \geq 1$$
and so $$ |\frac{(m-1)(n-1)}{mn}|^{|\frac {(m-1)(n-1)}{m-n}|} \leq
1,$$ and $$|\frac{m-n}{mn}|^{|\frac {m+n-1}{m-n}|} \leq
\frac{1}{2}.$$\\
Therefore for this case (2.13) is holds. \\

\textbf{Case 2}. Let $x=\frac{1}{m} , y=\frac{1}{n}$ and $m$ and
$n$ are odd.\\

\textbf{Case 3}. Let $x=\frac{1}{m} , y=\frac{1}{n}$, $m$ is even
and $n$ is odd.\\
By the same argument in case 1 we conclude that (2.13) for
case 2 and case 3 is holds.\\

\textbf{Case 4}. Let $x=0, y=\frac{1}{n}$ such that $n$ is even.
Then
\begin{align*}
 & \int_0^{|Tfx-Tfy|}\phi(t) dt \leq \frac{1}{2}{\int_0^{|Tx-Ty|}\phi(t)
 dt} \\ & \Leftrightarrow (\frac{1}{n})^{n} \leq
 \frac{1}{2}{(\frac{1}{n-1})^{n-1}} \\ & \Leftrightarrow
 (\frac{1}{n})^{n} (n-1)^{n-1}\leq \frac{1}{2} \\
 & \Leftrightarrow (\frac{n-1}{n})^{n-1}. \frac{1}{n} \leq
 \frac{1}{2}.
\end{align*}
The last inequality is holds, because,
$$ (\frac{n-1}{n})^{n-1} \leq 1 \qquad and \qquad \frac{1}{n} \leq
\frac{1}{2}.$$\\
Therefore (2.13) is true for this case.\\
\textbf{Case 5}. Let $x=0\ \and \ \ y=\frac{1}{n}$ such that $n$
is odd. By the same argument in case 4 we conclude that,
(2.13) is holds for this case.\\
Hence, (2.13) is holds for all $x,y\in X.$ Therefore the condition
of Theorem 2.5 are hold and so $f$ has a unique fixed point.
\end{exam}

Email:

S-Moradi@araku.ac.ir

A-Beiranvand@Arshad.araku.ac.ir


\begin{thebibliography}{MaHo}
\bibitem{a} S. Banach, {\it Sur Les Operations Dans Les Ensembles
Abstraits et Leur Application Aux E'quations Inte'grales}, Fund.
Math. 3(1922), 133-181(French).
\bibitem{b}A. Beiranvand, S. Moradi, M. Omid and H. Pazandeh
, {\it Two Fixed-Point Theorem For Special Mapping},
to appear.\\
\bibitem{c} A. Branciari, {\it A fixed point theorem for mapping
satisfying a general contractive condition of integral type} Int.
J. M. and M. since, 29:9 (2002), 531-536.\\
\bibitem{d}R. Kannan, {\it Some results on fixed points,
Bull.Calcutta Math. Soc.} 60(1968),71-76.
\bibitem{e} Kazimierz Goebel and W. A. Kirk, Topiqs in Metric Fixed
Point Theory, Combridge University Press, New York, 1990.
\bibitem{d} J. Meszaros, {\it A Comparison of Various Definitions of
Contractive Type Mappings}, Bull. Calcutta Math. Soc. 84(1992),
no. 2, 167-194.
\bibitem{f} B. E. Rhoades, {\it A Comparison of Various Definitions of
Contractive Mappings}, Trans. Amer. Math. Soc. 226(1977), 257-290.
\bibitem{g} B. E. Rhoades, {\it Contractive definitions revisited},
Topological Methods in Nonlinear Functional Analysis (Toronto,
Ont.,1982), Contemp. Math., Vol. 21, American Mathematical
Society, Rhode Island, 1983, pp. 189-203.
\bibitem{h} B. E. Rhoades, {\it Contractive Definitions}, Nonlinear
Analysis, World Science Publishing, Singapore, 1987, pp. 513-526.
\bibitem{i} O. R. Smart, Fixed Point Theorems, Cambridge University
Press, London, 1974.

\end{thebibliography}
\end{document}